\documentclass[12pt]{amsart}
\usepackage[margin=2cm]{geometry}
\usepackage{xcolor}
\usepackage{amssymb}
\usepackage{amsmath}
\usepackage{mathtools,soul}

\usepackage[colorlinks=true, linkcolor=red, citecolor=green, urlcolor=blue, pagebackref=false, breaklinks=true]{hyperref}

\newtheorem{theorem}{Theorem}[section]
\newtheorem{lemma}[theorem]{Lemma}
\newtheorem{corollary}[theorem]{Corollary}

\theoremstyle{definition}
\newtheorem{definition}[theorem]{Definition}

\newtheorem{conjecture}[theorem]{Conjecture}

\usepackage[capitalise]{cleveref}
\crefalias{remark}{remark}

\DeclareMathOperator{\NN}{\mathbb{N}}

\DeclareMathOperator{\Supp}{Supp}

\DeclareMathOperator{\ord}{ord}
\DeclareMathOperator{\Spec}{Spec}

\DeclareMathOperator{\Frac}{Frac}

\newcommand{\mf}{\mathfrak}
\newcommand{\mbb}{\mathbb}

\begin{document}

\title[Polynomial interpolation and Waldschmidt constant of points]{Polynomial interpolation and the Waldschmidt constant of points in projective space}

\author{T\`ai Huy H\`a}
\address{T\`ai Huy H\`a\\Tulane University, Department of Mathematics, 6823 St. Charles Avenue, New Orleans, LA 70118, USA}
\email{tha@tulane.edu}

\author{Aniketh Sivakumar}
\address{Aniketh Sivakumar\\Tulane University, Department of Mathematics, 6823 St. Charles Avenue, New Orleans, LA 70118, USA}
\email{asivakumar@tulane.edu}



\begin{abstract}
We establish Chudnovsky's Conjecture and, more generally, Demailly's Conjecture for any finite set of points in $\mathbb{P}^n_{k}$, where $k$ is an algebraically closed field of characteristic 0. The key idea of passing to positive characteristic, as well as the broad plan of the proof, was suggested by ChatGPT-5.6 Sol. All mathematical arguments and proofs in this paper were developed, written, and verified by the authors.
\end{abstract}

\maketitle

\section{Introduction}

In this paper, we prove in full generality long-standing conjectures of Chudnovsky and Demailly for any finite set of points in projective space. These conjectures concern lower bounds for the answer to the following fundamental problem in polynomial interpolation: \emph{given a finite set of distinct points $X\subset \mbb{P}^n_k$ and an integer $m\geq 1$, what is the least degree of a homogeneous polynomial that vanishes at each point of $X$ to order at least $m$?}

Let $S=k[x_0,\dots,x_n]$ be the homogeneous coordinate ring of $\mbb{P}^n_k$ and let $I=I(X)\subset S$ be the defining ideal of $X$. For a homogeneous ideal $J\subset S$, let $\alpha(J)$ denote its initial degree, that is, the least degree of a nonzero homogeneous polynomial in $J$. By the Zariski--Nagata Theorem \cite{Zar49,Nag62,EH79,DDSG+18}, when $k$ is perfect, the answer to the interpolation problem above is precisely $\alpha(I^{(m)})$, where $I^{(m)}$ denotes the $m$-th symbolic power of $I$.

A natural asymptotic invariant associated to this problem is the Waldschmidt constant
$$
\widehat{\alpha}(I):=\lim_{r\rightarrow\infty}\frac{\alpha(I^{(r)})}{r}.
$$
The limit always exists and, furthermore,
${\displaystyle \widehat{\alpha}(I)=\inf_{r\in\mbb{N}}\frac{\alpha(I^{(r)})}{r}}$;
see, for example, \cite[Lemma 2.3.1]{BH10}. Thus, the Waldschmidt constant measures the asymptotic growth of the least degree of a polynomial vanishing to a prescribed order at all the points of $X$.

The study of these numbers goes back to work of Waldschmidt and Skoda \cite{Wal77,Sko77} in complex analysis and Diophantine approximation. Chudnovsky \cite{Chu81} and Demailly \cite{Dem82} raised the following conjectures.

\begin{conjecture} \label{conj}
Let $I \subseteq S$ be the defining ideal of a finite set of points in $\mbb{P}^n_{\mbb{C}}$. 
\begin{enumerate}
    \item (Chudnovsky's Conjecture) The Waldschmidt constant of $I$ satisfies the inequality
    $${\displaystyle \widehat{\alpha}(I)\geq \frac{\alpha(I)+n-1}{n}.}$$
    \item (Demailly's Conjecture) For all $m \ge 1$, one has
    $${\displaystyle \widehat{\alpha}(I)\geq \frac{\alpha(I^{(m)})+n-1}{m+n-1}.}$$
\end{enumerate}
\end{conjecture}

When $m = 1$, Demailly's Conjecture gives precisely Chudnovsky's Conjecture. Our main result proves both Demailly's and Chudnovsky's Conjectures without any assumption on either the number or the position of the points.

\begin{theorem} \label{demailly}
Suppose that $k$ is an algebraically closed field of characteristic zero and $X$ is a finite set of distinct points in $\mbb{P}^n_k$ with defining ideal $I=I(X)\subset S$. Then, for any $m\in\mbb{N}$,
$$
\widehat{\alpha}(I)\geq \frac{\alpha(I^{(m)})+n-1}{m+n-1}.
$$
\end{theorem}

The polynomial interpolation problem is classical and notoriously difficult. Even when the points are general, the numbers $\alpha(I^{(m)})$ are not known in general. The celebrated Alexander--Hirschowitz Theorem \cite{AH95} gives a complete solution for general double points in projective space. For higher multiplicities, many fundamental problems remain open. In the plane, these questions are closely related to the Segre--Harbourne--Gimigliano--Hirschowitz (SHGH) Conjecture, while their asymptotic behavior is related to Nagata's Conjecture and its higher-dimensional analogues; see, for instance, \cite{Nag59,Iar97,HM21} and the references therein. The study of these questions is also naturally connected to symbolic powers and their asymptotic invariants; see, for example, \cite{BH10,CHHVT20,DDSG+18,ELS01,BHH13,HH02}.

There has been substantial progress toward Chudnovsky's and Demailly's Conjectures. Chudnovsky himself established his bound for an arbitrary finite set of points in $\mbb{P}^2_{\mbb{C}}$ \cite{Chu81}, and Esnault and Viehweg \cite{EV83} proved Demailly's Conjecture for an arbitrary finite set of points in $\mbb{P}^2_{\mbb{C}}$. In higher dimensions, most previous results on these conjectures have required the points to be \emph{general} or \emph{very general}. 

For Demailly's Conjecture, Malara, Szemberg and Szpond \cite{MSS18} proved that, for a fixed $m$, the conjecture holds for sufficiently many very general points. Chang and Jow \cite{CJ20} improved these results and, in particular, proved the conjecture for a perfect $n$-th power number of very general points. Nagel and Trok \cite{NT19} established the conjecture for at most $n+2$ general points and for certain cases of $n+3$ general points. In \cite{BGHN22}, Bisui, Grifo, H\`a and Nguy$\tilde{\text{\^e}}$n proved Demailly's Conjecture for sufficiently many general points. Dumnicki, Szemberg and Szpond \cite{DSS24} subsequently proved the conjecture for at least $m^n$ very general points. More recently, Bisui and Nguy$\tilde{\text{\^e}}$n \cite{BN26} settled the case $m=2$ for very general points and obtained further results for general points, while Bisui and Mahato \cite{BM24} obtained improved bounds for general points and investigated the case $m=3$.

Chudnovsky's Conjecture was also established for various classes of very general and general points. For instance, Dumnicki \cite{Dum15}, Dumnicki and  Tutaj-Gasi\'nska \cite{DTG17} proved the conjecture for sufficiently many very general points. Fouli, Mantero and Xie removed the condition on the number of points and established the conjecture for very general points. Bisui, Grifo, H\`a and Nguy$\tilde{\text{\^e}}$n \cite{BGHN22b} proved the conjecture for sufficiently many general points. More recently, Bisui and Nguy$\tilde{\text{\^e}}$n \cite{BN24} completed the proof for general points. 

Special configurations, such as star configurations, have also played an important role in the development of the subject; see \cite{BH10,GHM13}. Nevertheless, before the present work, in dimension $n\geq 3$, both Chudnovsky's and Demailly's Conjectures remained open for an arbitrary finite set of points.

Our proof of Theorem~\ref{demailly} takes a different approach from previous work. Rather than imposing conditions on the number or the position of the points, our main idea is to first establish Demailly's inequality in sufficiently large positive characteristic, where we can make use of the interaction between differential operators and the Frobenius map. We then use specialization to descend to characteristic zero.

By choosing a hyperplane that contains none of the points and de-homogenization map, we can move the interpolation from projective space to affine space.
We then work over an algebraically closed field of characteristic $p\gg 0$. The differential operators $d_\lambda$ that we use are Hasse derivatives, and they allow us to control the order of vanishing of a polynomial at each of the points. On the other hand, if $q=p^e$, the Frobenius map gives a decomposition of a polynomial of the form
$$
f=\sum_\lambda h_\lambda^q\mathbf{x}^{\lambda}.
$$
The key observation is that, by applying a suitable differential operator to $f$, one can isolate one of the coefficients $h_\lambda^q$ while controlling its order of vanishing at all of the points.

This leads to the main estimate in positive characteristic. Fix $m\geq 1$ and let
$$
M=q(m+n-1)-(n-1),
$$
where $q=p^e$. We prove in Theorem \ref{mainineq} that, when $p$ is sufficiently large,
$$
\alpha(J^{(M)})\geq q\alpha(J^{(m)})+(q-1)(n-1).
$$
Dividing by $M$ and letting $q\rightarrow\infty$ gives
$$
\widehat{\alpha}(J)\geq \frac{\alpha(J^{(m)})+n-1}{m+n-1}.
$$
It is worth noting that the finite inequality above is stronger than the asymptotic statement needed for Demailly's Conjecture. When $m=1$, Theorem \ref{mainineq} also establishes the positive characteristic analogue of a conjecture of Harbourne and Huneke, stated in \cite[Conjecture 4.8]{BHH13}, for a Frobenius subsequence $r = p^e$.

The last step of the proof is to specialize from characteristic zero to positive characteristic while keeping track of the interpolation numbers. Starting from the coordinates of the given points, we construct a finitely generated $\mbb{Z}$-algebra and choose a suitable fiber in arbitrarily large positive characteristic. The specialization is constructed so that the points remain distinct and, for the fixed value $m$ and every $r \geq 1$,
$$
\alpha(I(X)^{(m)})=\alpha(I(Y)^{(m)}) \text{ and }
\alpha(I(X)^{(r)})\geq \alpha(I(Y)^{(r)}).
$$
It follows that
$$
\widehat{\alpha}(I(X))\geq\widehat{\alpha}(I(Y)).
$$
Applying our positive characteristic result to $Y$ then gives Theorem~\ref{demailly}. An important feature of this specialization argument is that no genericity assumption on the points is needed.

The paper is outlined as follows. In Section~2, we pass from projective space to affine space and show that de-homogenization preserves the initial degrees of symbolic powers. In Section~3, we introduce the differential operators and Frobenius decomposition that will be used in the proof, and establish Demailly's Conjecture for finite sets of affine points over algebraically closed fields of sufficiently large positive characteristic. In Section~4, we construct the specialization from characteristic zero to positive characteristic and complete the proof of Theorem~\ref{demailly}.


\subsection*{Role of AI in this work}

ChatGPT-5.6 Sol was used in this project. In particular, it suggested a positive-characteristic approach to the problem, first proving the desired inequality for affine points in this setting and then passing to characteristic zero by specialization. It also outlined proofs of key theorems, including Theorem \ref{mainineq}, using Hasse derivatives in the setting of Chudnovsky's conjecture, i.e., when $m = 1$. These suggestions played an important role in the approach taken in this paper.

The mathematical arguments and proofs presented here were subsequently developed, written, and independently verified by the authors. The authors take full responsibility for the correctness of all results and arguments in the paper.


\subsection*{Acknowledgment.} The use of ChatGPT is indispensable in this work. The first author is partially supported by a Simons Foundation grant.


\section{Projective to affine reduction}
Throughout this section, let $S = k[x_0,\dots, x_n]$ be the polynomial ring in $n+1$ variables over an algebraically closed field $k$. Let $\mathfrak{m}$ denote the homogeneous maximal ideal in $S$. We recall the following.



Our first step to proving Theorem \ref{demailly} is to move to ``equivalent'' points in affine space. Let $X = \{P_1,\dots, P_r\}$ be a finite set of distinct points in $\mbb{P}^n$ with $P_i = [p_{i0}:\cdots:p_{in}]$ and $I(X) = \cap_{i=1}^r I(P_i)$.




Without loss of generality, via an appropriate change of coordinates, we may assume that none of the points in $X$ lies on the hyperplane $x_0 = 0$. That is, $X = \{P_1,\dots P_s\}$ with $p_{i0}=1$ for all $1\leq i\leq s$. Let $R = k[y_1,\dots, y_n]$ and $\mf{m}_i = (y_1-p_{i1},\dots, y_n-p_{in})$; in other words, the de-homogenization of $I(P_i)$ with respect to $x_0$.

\begin{theorem}\label{affinepoints}
Consider the ideal $J = \cap_{i=1}^s \mf{m}_i\subset R$. Then,
\[\alpha(I(X)^{(r)}) = \alpha(J^{(r)}).\]
for all $r\geq 1$.
    
\end{theorem}

\begin{proof}
    Let $z_{ij} = x_j-p_{ij}x_0$ and $w_{ij} = y_j-p_{ij}$. For each $d\geq 1$, de-homogenization with respect to $x_0$ gives an isomorphism of vector spaces
    $$\phi_d:S_d\rightarrow R_{\leq d}$$
    $$F\rightarrow F(1,y_1,\dots,y_n).$$
    Furthermore, if $f\in S_e$ and $g\in S_{d-e}$ then $\phi_d(fg) = \phi_e(f)\phi_{d-e}(g)$. Consider a form $F\in I(X)^{(r)}$ of degree $d$. Since $I(X)$ is an ideal of points, $I(X)^{(r)} = \cap_{i=1}^sI(P_i)^r$. Therefore, for each $i$, $F = \sum_{|\lambda|=r} c_{i\lambda}z_i^{\lambda}$, where $c_{i\lambda}\in S_{d-r}$ and $\lambda\in\mbb{N}^n$. It follows that
    $$\phi_d(F) = \sum_{|\lambda|=r} \phi_{d-r}(c_{i\lambda})\phi_r(z_i^{\lambda})
    = \sum_{|\lambda|=r} \phi_{d-r}(c_{i\lambda})\prod_{j=1}^n\phi_1(z_{ij})^{\lambda_j}.$$
    Since $\phi_1(z_{ij}) = w_{ij}$, we have
    $$\phi_d(F) = \sum_{|\lambda|=r} \phi_{d-r}(c_{i\lambda})w_i^{\lambda}\in \mf{m}_i^r$$
    for all $i$. Hence, $\phi_d(F)\in J^{(r)}$. Since $\deg(\phi_d(F))\leq \deg(F)$, we have
    $$\alpha(I(X)^{(r)})\geq \alpha(J^{(r)}).$$

    Now we use the homogenization operation
    $$\psi_d:R_{\leq d}\rightarrow S_d \text{ defined by } f \rightarrow x_0^df\left(\frac{x_1}{x_0},\dots,\frac{x_n}{x_0}\right).$$
    Observe that this is $\phi_d^{-1}$ and has the property that if $\deg(f)=r$ and $\deg(g)=t-r$ for $t\leq d$, then
    $$\psi_d(fg)=x_0^{d-t}\psi_r(f)\psi_{t-r}(g).$$

    Choose $0\neq f\in J^{(r)}$ such that
    $\deg(f)=\alpha(J^{(r)})=:t.$
    Fix $i$. Since $f\in\mf{m}_i^r$, writing $f$ in the coordinates
    $w_{i1},\dots,w_{in}$ gives
    $$f=\sum_{\substack{\nu\in\mbb{N}^n\\ r\leq|\nu|\leq t}}
    a_{i\nu}w_i^\nu.$$
    For each $\nu$ occurring in this sum, choose
    $\lambda\leq\nu$ coordinatewise with $|\lambda|=r$. Then
    $w_i^\nu=w_i^\lambda w_i^{\nu-\lambda}$ and
    $|\nu-\lambda|\leq t-r$. Grouping the terms according to
    $\lambda$, we obtain
    $$f=\sum_{|\lambda|=r}c_{i\lambda}w_i^\lambda \text{ with } \deg(c_{i\lambda})\leq t-r.$$

    Since $\psi_t$ is linear,
    $$\psi_t(f)=\sum_{|\lambda|=r}\psi_t(c_{i\lambda}w_i^\lambda).$$
    If $\deg(c_{i\lambda})=a_\lambda$, then $a_\lambda+r\leq t$, and hence
    $$\psi_t(f)
    =\sum_{|\lambda|=r}
    x_0^{t-a_\lambda-r}\psi_{a_\lambda}(c_{i\lambda})z_i^\lambda.$$
    Therefore, $\psi_t(f)\in I(P_i)^r$ for every $i$, and thus
    $\psi_t(f)\in I(X)^{(r)}$. Since $\psi_t(f)\neq0$ and
    $\deg\psi_t(f)=t$, we obtain
    $$\alpha(I(X)^{(r)})\leq t=\alpha(J^{(r)}).$$
    We obtain the desired equality.
\end{proof}

Since the initial degrees are the same, it suffices to prove Demailly's conjecture in the affine setting introduced above.

\section{Positive characteristic and Frobenius estimate}

We shall now prove Demailly's conjecture in the affine setting for polynomial rings over algebraically closed fields of characteristic $p \gg 0$. 

By defining 
$$\alpha(J) = \min \{\deg f \mid 0 \not= f \in J\}$$
for any ideal $J \subseteq R = k[x_1, \dots, x_n]$, we can also discuss $\alpha(J^{(t)})$ and the Waldschmidt constant of $J$ as those of $I \subseteq S = k[x_0, \dots, x_n]$ as in the introduction.

For polynomial rings over a field of any characteristic, we shall make use of Hasse derivatives.

\begin{definition}\label{diffmon}
Let $\beta,\lambda\in\NN^n$ and let $R=k[x_1,\dots,x_n]$. The
$\lambda$-th Hasse derivative is the $k$-linear operator
$d_\lambda:R\longrightarrow R$ defined on monomials by
$$
d_\lambda(\mathbf{x}^\beta)
=
\begin{cases}
\displaystyle
\binom{\beta_1}{\lambda_1}\cdots
\binom{\beta_n}{\lambda_n}\mathbf{x}^{\beta-\lambda},
& \text{if $\beta\geq\lambda$ coordinatewise},\\
0, & \text{otherwise}.
\end{cases}
$$
\end{definition}

We shall use the following standard properties of Hasse derivatives.
They satisfy the product formula
$$
d_\lambda(fg)
=
\sum_{\alpha+\beta=\lambda}
d_\alpha(f)d_\beta(g).
$$
They also satisfy the Hasse--Taylor formula. For
$\mathbf{a}=(\mathbf{a}_1,\dots,\mathbf{a}_n)\in R^n$, let
$\tau_{\mathbf{a}}:R\longrightarrow R$ denote the substitution map
$x_i\mapsto\mathbf{a}_i$. Then, for $\mathbf{y},\mathbf{h}\in R^n$,
$$
\tau_{\mathbf{y}+\mathbf{h}}(f)
=
\sum_{\lambda\in\NN^n}
\tau_{\mathbf{y}}(d_\lambda(f))\mathbf{h}^\lambda.
$$
In particular, if $a\in k^n$, taking $\mathbf{y}=a$ and
$\mathbf{h}=x-a$ gives
$$
f
=
\sum_{\lambda\in\NN^n}
\tau_a(d_\lambda(f))(x-a)^\lambda.
$$

\begin{lemma}\label{diffdecrease}
    Let $\mf{m}=(x_1-a_1,\dots,x_n-a_n) \subset R = k[x_1,\dots x_n]$, for some $a = (a_1, \dots, a_n) \in k^n$. Then,
    \[d_\lambda(\mf{m}^s)\subseteq \mf{m}^{\max\{s-|\lambda|,0\}} \]
    Furthermore,
    \[f\in \mf{m}^s \iff  d_\lambda(f)\in \mf{m} \text{ for all $|\lambda|< s $}\]
\end{lemma}

\begin{proof}
For $\beta\in\NN^n$, we have
$$
d_\lambda((x-a)^\beta)
=
\binom{\beta}{\lambda}(x-a)^{\beta-\lambda}
$$
if $\beta\geq\lambda$ coordinatewise, and
$d_\lambda((x-a)^\beta)=0$ otherwise. Hence, if $|\beta|\geq s$ and
$d_\lambda((x-a)^\beta)\neq0$, then $|\beta-\lambda|=|\beta|-|\lambda|\geq s-|\lambda|.$
Therefore,
$$
d_\lambda((x-a)^\beta)
\in
\mf{m}^{\max\{s-|\lambda|,0\}}.
$$
Since $\mf{m}^s$ is generated by the monomials $(x-a)^\beta$ with
$|\beta|\geq s$, this proves the first assertion.

For the second assertion, consider $y=a$ and $h=x-a$. As seen before, ${\displaystyle f=\sum_{\lambda\in\NN^n} \tau_a(d_\lambda(f))(x-a)^\lambda.}$ 
If $f\in\mf{m}^s$, then by the first assertion
$d_\lambda(f)\in\mf{m}$ for every $|\lambda|<s$.

Conversely, suppose that $d_\lambda(f)\in\mf{m}$ for every
$|\lambda|<s$. Since $\mf{m}=\ker\tau_a$, we have
$\tau_a(d_\lambda(f))=0$ for every $|\lambda|<s$. Thus, in the
Taylor expansion above, every term of total degree less than $s$
vanishes. Hence, $f\in\mf{m}^s$.
\end{proof}

\begin{lemma}\label{deripower}
    Let $q=p^e$ and $h\in R$. Then,
    \[d_\lambda(h^q) = 0 \text{ unless } \lambda \in q\mbb{N}^n. \]
    Furthermore, if $\lambda = q\mu$ with $\mu \in \mbb{N}^n$, then we have $d_\lambda(h^q) = (d_\mu(h))^q$
\end{lemma}

\begin{proof}
    For any $a\in k^n\subset R^n$, consider,
    \[\tau_{\mathbf{x}+a}(h^q) = \big(\sum_{\lambda\in \mbb{N}^n}\tau_a(d_\lambda(h))\mathbf{x}^\lambda\big)^q = \sum_{\lambda\in \mbb{N}^n}\tau_a(d_\lambda(h)^q)\mathbf{x}^{q\lambda}.\]
    By uniqueness of multigraded components, this decomposition is unique. Since $\mf{J}(R) = 0$, $\tau_a(f)=0$ for all $a\in k^n$ if and only if $f=0$. Hence, the proposition follows.
\end{proof}

We recall some basic definitions and notation for rings of characteristic $p$.

\begin{definition}[Frobenius map]\label{frobenius}
Let $A$ be a ring of characteristic $p>0$ and let $q=p^e$. The
$e$-th iterated Frobenius map is the ring homomorphism
$$
F^e:A\longrightarrow A,\qquad r\longmapsto r^q.
$$
We denote by $F_*^e(A)$ the $A$-module obtained from $A$ by restriction
of scalars along $F^e$. Thus, as an abelian group $F_*^e(A)=A$, and
the $A$-module structure is given by
$$
a\cdot F_*^e(r)=F_*^e(a^q r)
$$
for $a,r\in A$.
\end{definition}

The following folklore result on the structure of the polynomial ring over a field of characteristic $p$ is helpful.

\begin{lemma}\label{polybasis}
Let $k$ be a perfect field of characteristic $p>0$, let $q=p^e$, and
let $R=k[x_1,\dots,x_n]$. Set $\Lambda_q=\{0,\dots,q-1\}^n.$
Then $F_*^e(R)$ is a free $R$-module with basis
$$
\left\{F_*^e(\mathbf{x}^{\lambda})\mid
\lambda\in\Lambda_q\right\}.
$$
Equivalently, every $f\in R$ can be written uniquely in the form
$$f=\sum_{\lambda\in\Lambda_q}h_\lambda^q\mathbf{x}^{\lambda}, \text{ where } h_\lambda\in R.$$
\end{lemma}

\begin{proof}
Every $\beta\in\mbb{N}^n$ can be written uniquely as $\beta=q\mu+\lambda$
with $\mu\in\mbb{N}^n$ and $\lambda\in\Lambda_q$. Since $k$ is
perfect, every $c\in k$ can be written uniquely as $c=a^q$ for some
$a\in k$. Hence,
$$
F_*^e(c\mathbf{x}^{\beta})
=
F_*^e(a^q\mathbf{x}^{q\mu+\lambda})
=
a\mathbf{x}^{\mu}\cdot F_*^e(\mathbf{x}^{\lambda}).
$$
It follows that the elements
$F_*^e(\mathbf{x}^{\lambda})$, $\lambda\in\Lambda_q$, generate
$F_*^e(R)$ as an $R$-module.

To see that they are linearly independent, suppose that
$$
\sum_{\lambda\in\Lambda_q}
h_\lambda\cdot F_*^e(\mathbf{x}^{\lambda})=0.
$$
By the definition of the $R$-module structure on $F_*^e(R)$, this is
equivalent to
$$
F_*^e\left(
\sum_{\lambda\in\Lambda_q}
h_\lambda^q\mathbf{x}^{\lambda}\right)=0,
$$
and hence
$$
\sum_{\lambda\in\Lambda_q}
h_\lambda^q\mathbf{x}^{\lambda}=0.
$$
For distinct $\lambda\in\Lambda_q$, the monomials occurring in
$h_\lambda^q\mathbf{x}^{\lambda}$ have exponent vectors belonging to
distinct residue classes modulo $q$. Thus, the supports of these
summands are pairwise disjoint, and it follows that $h_\lambda=0$
for every $\lambda$. This proves that the displayed set is a basis.

The equivalent polynomial decomposition follows immediately from
the $R$-module decomposition of $F_*^e(R)$.
\end{proof}

For a polynomial
$f=\sum_{\lambda\in\Lambda_q}h_\lambda^q\mathbf{x}^\lambda$
as in Lemma~\ref{polybasis}, set
$$
\Supp(f)=\{\lambda\in\Lambda_q\mid h_\lambda\neq 0\}.
$$

\begin{lemma}\label{diffcoeff}
Let
$$
f=\sum_{\lambda\in\Lambda_q}h_\lambda^q\mathbf{x}^\lambda\in R
$$
be the decomposition of $f$ given by Lemma~\ref{polybasis}. If
$\gamma$ is a maximal element of $\Supp(f)$ with respect to the
usual partial ordering on $\mbb{N}^n$, then
$$
d_\gamma(f)=h_\gamma^q.
$$
\end{lemma}

\begin{proof}
Fix $\lambda\in\Supp(f)$. By the product formula,
$$
d_\gamma(h_\lambda^q\mathbf{x}^\lambda)
=
\sum_{\alpha+\beta=\gamma}
d_\alpha(h_\lambda^q)d_\beta(\mathbf{x}^\lambda).
$$
By Lemma \ref{deripower}, 
$d_\alpha(h_\lambda^q)$ can be nonzero only if
$\alpha\in q\mbb{N}^n$. On the other hand, $\alpha\leq\gamma$
coordinatewise, and $\gamma\in\Lambda_q$, so
$0\leq\gamma_i<q$ for every $i$. It follows that $\alpha=0$.
Therefore,
$$
d_\gamma(h_\lambda^q\mathbf{x}^\lambda)
=
h_\lambda^q d_\gamma(\mathbf{x}^\lambda).
$$

By Definition~\ref{diffmon}, $d_\gamma(\mathbf{x}^\lambda)=0$
unless $\lambda\geq\gamma$ coordinatewise. Since $\gamma$ is
maximal in $\Supp(f)$, the only $\lambda\in\Supp(f)$ satisfying
$\lambda\geq\gamma$ is $\lambda=\gamma$. Moreover,
$d_\gamma(\mathbf{x}^\gamma)=1$. Hence,
$$
d_\gamma(f)=h_\gamma^q,
$$
as desired.
\end{proof}

We now return to the setting of ideals generated by points. Let $J$ be as defined in the prior section (as in Theorem \ref{affinepoints}), i.e., $J$ is the defining ideal of a finite set of distinct points in $\mathbb{A}^n_k$.

\begin{lemma}\label{tbound}
Fix $m\geq 1$. Assume that $p>\alpha(J^{(m+n-1)}).$
Then, for every $t$ with $m\leq t\leq m+n-1$,
$$
\alpha(J^{(t)})
\geq
\alpha(J^{(m)})+t-m.
$$
\end{lemma}

\begin{proof}
We proceed by induction on $t$. The assertion is clear when $t=m$.
Suppose that $m<t\leq m+n-1$ and that
$$
\alpha(J^{(t-1)})
\geq
\alpha(J^{(m)})+t-m-1.
$$

Choose a nonzero polynomial $f\in J^{(t)}$ of minimal degree, so that
$$
\deg(f)=\alpha(J^{(t)}).
$$
Since $t\leq m+n-1$, we have
$J^{(m+n-1)}\subseteq J^{(t)}$, and therefore
$$
\deg(f)=\alpha(J^{(t)})
\leq
\alpha(J^{(m+n-1)})<p.
$$

We claim that $d_{e_i}(f)\neq 0$ for some $1\leq i\leq n$. Indeed,
if $d_{e_i}(f)=0$ for every $i$, then every monomial
$\mathbf{x}^\beta$ occurring in $f$ has each $\beta_i$ divisible
by $p$. Thus,
$$
f\in k[x_1^p,\dots,x_n^p].
$$
Since $f\in J^{(t)}$ and $t\geq1$, the polynomial $f$ is not a
nonzero constant. It would then follow that $\deg(f)\geq p$,
contrary to $\deg(f)<p$.

Choose $i$ such that $d_{e_i}(f)\neq0$. By
Lemma~\ref{diffdecrease},
$$
d_{e_i}(f)\in J^{(t-1)}.
$$
Thus, by the induction hypothesis,
$$
\deg d_{e_i}(f)
\geq
\alpha(J^{(t-1)})
\geq
\alpha(J^{(m)})+t-m-1.
$$
On the other hand, since $d_{e_i}(f)\neq0$,
$$
\deg d_{e_i}(f)\leq \deg(f)-1.
$$
Consequently,
$$
\alpha(J^{(t)})
=
\deg(f)
\geq
\alpha(J^{(m)})+t-m,
$$
which proves the assertion.
\end{proof}

The following result establishes the key inequality needed for Demailly's Conjecture. When $m = 1$, it also proves the inequality predicted by Harbourne–Huneke's conjecture, \cite[Conjecture 4.8]{BHH13}, for the Frobenius subsequence $r = p^e$. 

\begin{theorem}\label{mainineq}
    Fix $m\geq 1$. Let $k$ be an algebraically closed field of characteristic $p> \alpha(J^{(m)})+(n-1)\alpha(J)$. Define $M = q(m+n-1)-(n-1)$, where $q=p^e$ for $e \ge 1$. Then,
    \[\alpha(J^{(M)})\geq q\alpha(J^{(m)})+(q-1)(n-1).\]
\end{theorem}

\begin{proof}
    Since $J^{n-1}J^{(m)}\subseteq J^{(m+n-1)}$, $$p> \alpha(J^{(m)})+(n-1)\alpha(J)\geq \alpha(J^{(m+n-1)}).$$

    Let $0 \not= f\in J^{(M)} = \cap_{i=1}^s \mf{m}_i^M$ be a form of degree $d$. Let $f = \sum_{\lambda\in \Lambda_q}h_\lambda^q \mathbf{x}^\lambda$ be the unique decomposition of $f$ as in Lemma \ref{polybasis}. Consider $\gamma \in \Supp(f)$, a maximal element with respect to the partial ordering in $\mbb{N}^n$. Then, by Lemma \ref{diffdecrease} and Lemma \ref{diffcoeff}, $h_\gamma^q\in \mf{m}_i^{M-|\gamma|}$ for all $i$. Since $R$ is a polynomial ring over a field, $gr_{\mf{m}_i}(R) = \bigoplus_{j=0}^\infty \mf{m}_i^j/\mf{m}_i^{j+1}$ is a reduced ring. Hence, $\ord_{\mf{m}_i}(h_\gamma^q) = q\ord_{\mf{m}_i}(h_\gamma)$ for all $i$. Therefore,
    \[h_\gamma\in \bigcap_{i=1}^s \mf{m}_i^{t} = J^{(t)}\text{\hspace{3mm} where $t = \left\lceil\frac{M-|\gamma|}{q}\right\rceil$}\]

    Note that $|\gamma|\leq n(q-1)$. Hence, $M-|\gamma|\geq q(m-1)+1$. Therefore, $m+n-1\geq t\geq m$. From Lemma \ref{tbound}, 
    $$\deg(h_\gamma)\geq \alpha(J^{(m)})+t-m.$$
    Hence, $\deg(h_\gamma^q)\geq q\alpha(J^{(m)})+q(t-m)$. Moreover, by the uniqueness of the decomposition in Lemma \ref{polybasis}, the monomial supports of the summands $h_\lambda^q\mathbf{x}^\lambda$, for $\lambda\in\Lambda_q$, lie in distinct residue classes modulo $q$. Hence, there can be no cancellation between different summands, and therefore
    $$
d=\deg(f)\geq \deg(h_\gamma^q\mathbf{x}^\gamma)
=q\deg(h_\gamma)+|\gamma|.
$$
By definition, $qt\geq M-|\gamma|$. Hence,
    $$d\geq q\alpha(J^{(m)})+qt-qm+|\gamma|\geq q\alpha(J^{(m)})+M-qm = q\alpha(J^{(m)})+(q-1)(n-1),$$
    as desired.
\end{proof}


\begin{corollary}[Demailly's Conjecture in Characteristic $p$]\label{demaillycharp}
    Let $k$ be an algebraically closed field of characteristic $p$, where $p$ is as in Theorem \ref{mainineq}. Then,

    \[\widehat{\alpha}(J)\geq \frac{\alpha(J^{(m)})+n-1}{m+n-1}.\]
\end{corollary}
\begin{proof}
    For all $q=p^e$, from Theorem \ref{mainineq}, 
    \[\frac{\alpha(J^{(M)})}{M} \geq \frac{q\alpha(J^{(m)})+(q-1)(n-1)}{M} = \frac{q\alpha(J^{(m)})+(q-1)(n-1)}{q(m+n-1)-(n-1)}.\]
    We can now evaluate the limit as $q\rightarrow \infty$ to obtain,
    \[\widehat{\alpha}(J) = \lim_{q\rightarrow \infty}\frac{\alpha(J^{(M)})}{M} \geq\lim_{q\rightarrow \infty}\frac{q\alpha(J^{(m)})+(q-1)(n-1)}{q(m+n-1)-(n-1)} = \frac{\alpha(J^{(m)})+n-1}{m+n-1}. \qedhere \]
\end{proof}


\section{Specialization and Demailly's Conjecture in characteristic 0}
In this final section, we shall descend Corollary \ref{demaillycharp} to characteristic 0 to obtain the proof for Theorem \ref{demailly}.

\begin{lemma}\label{fieldext}
Let $K \subseteq L$ be a field extension and let $\{Q_1,...,Q_r\}\in \mbb{A}^n_K$ be distinct points. Let $\mf{n}_i$ be the ideal of $Q_i$ in $K[y_1,...,y_n]$, and let $\mf{n}_{i,L}$ be its extension to $L[y_1,...,y_n]$. Then, for every $s\geq 1$,
\[\min\{\deg f\mid 0\neq f\in \bigcap_{i=1}^r \mf{n}^s_i\} = \min\{\deg f\mid 0\neq f\in \bigcap_{i=1}^r \mf{n}^s_{i,L}\}\]
\end{lemma}

\begin{proof}
    Consider the evaluation map
    \[\phi_{s,d}^{K}:K[y_1,\dots y_n]_{\leq d}\rightarrow \bigoplus_{\substack{1\leq i \leq r\\ |\beta|<s}}K\]
    \[f\rightarrow (d_\beta(f)(P_i))_{i,\beta}\]
    By Lemma \ref{diffdecrease}, 
    $$\ker \phi^K_{s,d} = K[y_1,\dots,y_n]_{\leq d}\bigcap \big(\cap_{i=1}^r \mf{n}^s_i\big).$$ 
    Since $K\subseteq L$ are field extensions, they are faithfully flat. Hence, $\ker \phi^L_{s,d} = \ker \phi^K_{s,d}\otimes L$. Therefore, $\ker \phi^K_{s,d}=0$ if and only if $\ker \phi^L_{s,d}=0$.
\end{proof}

\begin{lemma}\label{maindescent}
    Let $k$ be an algebraically closed field of characteristic zero and let $X\subset \mathbb{A}_k^n$ be a finite set of distinct points. Fix positive integers $m$, $N_0$. Then, there exists an algebraically closed field $L$ of characteristic $p>N_0$ and a finite set of distinct points $Y\subset \mathbb{A}_L^n$ such that
    \[\alpha(I(X)^{(s)})\geq \alpha(I(Y)^{(s)}) \text{ for all $s\geq 1$},\]
    \[\alpha(I(X)^{(m)}) = \alpha(I(Y)^{(m)}).\]
\end{lemma}

\begin{proof}
    Let $X = \{P_1,\dots,P_r\}\subset \mathbb{A}_k^n$ where $P_i = (p_{i1},\dots p_{in})$. Consider the finitely generated $\mbb{Z}$-algebra $\mbb{Z}[p_{ij}\mid 1\leq i\leq r, 1\leq j\leq n]$. Our first step is constructing an appropriate algebra where the points remain distinct over all fibres of the map to $\mbb{Z}$.
    
    For each $P_i$ and $P_j$ there exists an index $l_{ij}$ such that $p_{il_{ij}}-p_{jl_{ij}}\neq 0$. Let $Q = \prod_{i\neq j}(p_{il_{ij}}-p_{jl_{ij}})$ and $B_0 = \mbb{Z}[p_{ij}\mid 1\leq i\leq r, 1\leq j\leq n]_Q$. After localization, no two points $P_i$ and $P_j$ will have equivalent coordinates over the same prime ideal of $B_0$. Hence, they will be distinct in $(B_0/\mathfrak{p})_{\mathfrak{p}}$, where $\mathfrak{p}$ is any prime ideal in $B_0$. Consider the evaluation map:
    \[\phi_{s,d}^{B_0}:B_0[x_1,\dots x_n]_{\leq d}\rightarrow \bigoplus_{\substack{1\leq i \leq r\\ |\beta|<s}}B_0\]
    \[f\rightarrow (d_\beta(f)(P_i))_{i,\beta}\]
    Let $b_j = \alpha(I(X)^{(j)})$. By definition, $\phi_{m,b_m-1}^{B_0}\otimes_{B_0} k$ is injective. Hence, there exists a non-zero maximal minor $\Delta$ of the matrix of $\phi^{B_0}_{m,b_m-1}$. Set 
    $$B:=(B_0)_\Delta.$$
    Now $B$ is a finitely generated $\mbb{Z}$-algebra where the induced map $\phi_{m,b_m-1}^{B/\mathfrak{q}}$ has a non-zero maximal minor for any prime ideal $\mathfrak{q}\subset B$.

    We proceed to isolate an appropriate fiber of $B$ whose residue field has the desired properties. Since $B,\mathbb{Z}$ are Noetherian, and the morphism $\psi: \Spec(B)\rightarrow \Spec(\mbb{Z})$ is of finite type, by Chevalley's theorem, the image of $\psi$ is constructible. Furthermore, the generic point $0$ of $\Spec(\mbb{Z})\in im(\psi)$. Since $\Spec(\mbb{Z})$ is irreducible this implies that $im(\psi)$ contains an open dense subset of $\Spec(\mbb{Z})$. Therefore, all except finitely many primes are in $im(\psi)$. Choose 
    $$p>N_0\in im(\psi).$$ 
    It can be seen that $(B/pB)_p$ is a finitely generated $\mathbb{F}_p$-algebra with $\Spec(B_p/pB_p)\neq \phi$. Hence, $B_p$ has a non-zero maximal ideal $\mf{m}$ containing $pB_p$. Let $K =B_p/\mf{m}$.

    Let $\overline{p_{ij}}$ be the coordinates of the points $P_i$ in $K$. Define $Y_K=\{Q_1,\dots ,Q_r\}\subset \mathbb{A}_K^n$, where $Q_i = (\overline{p_{i1}},\dots, \overline{p_{in}})$. From our construction, all points $Q_i$ are distinct. Furthermore, $\phi_{m,b_m-1}^{K}$ has a non-zero maximal minor. Since $K$ is a field, $\phi_{m,b_m-1}^K$ is injective. Hence, 
    $$\alpha(I(Y_K)^{(m)})\geq b_m.$$

    Since $\Frac(B)\subseteq k$, by Lemma \ref{fieldext}, for $t\geq1$,
    $$\ker \phi_{t,b_t}^{k}\neq 0 \text{ and }\ker \phi_{t,b_t}^{\Frac(B)}\neq 0.$$ 
    Therefore, by multiplying by the common denominator of the coefficients of some $f\in \ker \phi_{t,b_t}^{\Frac(B)}$, we get that $\ker\phi_{t,b_t}^B\neq 0$ and thus $\ker\phi_{t,b_t}^{B_p}\neq 0$. Since the rank of $\phi_{t,b_t}^{B_p}$ cannot increase after quotienting, 
    $$\ker\phi_{t,b_t}^{K}\neq 0.$$ 
    This implies that $\alpha(I(Y_K)^{(t)})\leq b_t$ for all $t\geq 1$. In particular, 
    $$\alpha(I(Y_K)^{(m)}) = \alpha(I(X)^{(m)}).$$

    Finally, take $L = \overline{K}$ and $Y_L = \{Q_1,\dots, Q_r\}$. By Lemma \ref{fieldext},
    \[\alpha(I(X)^{(s)})\geq \alpha(I(Y_K)^{(s)}) = \alpha(I(Y_L)^{(s)}) \text{ for all $s\geq 1$},\]
    \[\alpha(I(X)^{(m)}) = \alpha(I(Y_K)^{(m)}) = \alpha(I(Y_L)^{(m)}). \qedhere \]
\end{proof}

We now have all the ingredients necessary to prove Theorem \ref{demailly}.

\begin{proof}[Proof of Theorem \ref{demailly}]
    Fix $m\geq 1$. By Theorem \ref{affinepoints}, For any set of points $X = \{P_1,\dots P_r\}\subset \mbb{P}^n$ with ideals $I(P_i)$, there exist corresponding points in $X' = \{P'_1,\dots, P'_r\} \subset \mathbb{A}_k^n$ with ideals $\mf{m}_i\subset k[x_1,\dots x_n]$ such that $\alpha(I(X)^{(s)}) = \alpha(\cap_{i=1}^r m_i^s) = \alpha(I(X')^{(s)}).$ for all $s\geq 1$. By Lemma \ref{maindescent}, there exists an algebraically closed field $L$ of characteristic $p>\alpha(I(X')^{(m)})+(n-1)\alpha(I(X'))$ and distinct points $Y =\{Q_1,\dots,Q_r\}\subset \mbb{A}_L^n$ such that
     \[\alpha(I(X')^{(s)})\geq \alpha(I(Y)^{(s)}) \text{ for all $s\geq 1$},\]
    \[\alpha(I(X')^{(m)}) = \alpha(I(Y)^{(m)}).\]
    
Since the first inequality above holds for all $s\geq 1$, taking $s=1$ gives
$\alpha(I(X'))\geq \alpha(I(Y))$. Therefore,
\[
\begin{aligned}
p
&>\alpha(I(X')^{(m)})+(n-1)\alpha(I(X'))\\
&=\alpha(I(Y)^{(m)})+(n-1)\alpha(I(X'))\\
&\geq \alpha(I(Y)^{(m)})+(n-1)\alpha(I(Y)).
\end{aligned}
\]
Thus, the hypothesis of Corollary \ref{demaillycharp} is satisfied for $Y$.
    Now by Corollary \ref{demaillycharp}, we have
    \[\widehat{\alpha}(I(X)) = \widehat{\alpha}(I(X')) \geq \widehat{\alpha}(I(Y))\geq \frac{\alpha(I(Y)^{(m)})+n-1}{m+n-1} = \frac{\alpha(I(X')^{(m)})+n-1}{m+n-1} = \frac{\alpha(I(X)^{(m)})+n-1}{m+n-1}.\]
\end{proof}

\bibliographystyle{amsplain}
\bibliography{refs}

\end{document}